\documentclass[12pt,a4paper]{amsart}
\usepackage{amsfonts}
\usepackage{amsthm}
\usepackage{amsmath}
\usepackage{amscd}
\usepackage[utf8]{inputenc}
\usepackage{bbold}
\usepackage{t1enc}
\usepackage[mathscr]{eucal}
\usepackage{indentfirst}
\usepackage{graphicx}
\usepackage{graphics}
\usepackage{pict2e}

\numberwithin{equation}{section}

\theoremstyle{plain}
\newtheorem{Th}{Theorem}[section]
\newtheorem{Lemma}[Th]{Lemma}
\newtheorem{Cor}[Th]{Corollary}

\newtheorem{Pro}[Th]{Proposition}

 \theoremstyle{definition}
\newtheorem{Def}[Th]{Definition}

\newtheorem{?}[Th]{Problem}

\DeclareMathOperator{\diag}{diag}

\DeclareMathOperator{\cha}{char}

\newcommand{\la}{\lambda }
\newcommand{\Z}{\mathbb{Z}}

\newcommand{\C}{\mathbb{C}}
\newcommand{\PP}{\mathbb{P}}
\newcommand{\Gr}{\mathbb{Gr}}

\newcommand{\tr}{\operatorname{tr}}

\begin{document}

\title{The triangulant}

\author[T.\ Bencze, P.\ E.\ Frenkel]{Tam\'as Bencze and P\'eter E.\ Frenkel}
\thanks{The second author's research is partially supported by NKFIH  grants K~146380 and KKP~139502.
}
\address{E\"{o}tv\"{o}s Lor\'{a}nd University,
  P\'{a}zm\'{a}ny P\'{e}ter s\'{e}t\'{a}ny 1/C, Budapest, 1117 Hungary \\ and R\'enyi Institute,  Budapest, Re\'altanoda u.\ 13-15, 1053 Hungary}
\email{ benczetamas11@gmail.com, frenkelp265@gmail.com}

 \dedicatory{}



\begin{abstract}We introduce the triangulant of two matrices, and relate it to the existence of orthogonal eigenvectors. We also use it for a new characterization of mutually unbiased bases.
Generalizing the notion, we introduce  higher order triangulants of two matrices, and relate them to the existence of nontrivially intersecting invariant subspaces of complementary dimensions. 
\end{abstract}

\maketitle
\section{Introduction}
There are many well-known results relating various  versions of (weak) commutativity of two (or more)  matrices to corresponding  versions of their simultaneous diagonalizability or triangularizability~\cite{L, RR}. In this paper, we consider a very weak version of simultaneous triangularizability --- so weak that matrix pairs satisfying it form a  hypersurface. Namely, we require the two matrices to have nontrivially intersecting invariant subspaces of complementary dimensions. Our main goal is to  determine the equation of that hypersurface. This equation expresses a  very weak version of commutativity.  When the ground field is $\C$, it  also yields a  new characterization of  mutually unbiased bases.

 \subsection*{Acknowledgement}
 
The authors are grateful to M\'aty\'as Domokos for helpful discussions.

\subsection*{Notations and terminology}

  The ring of $n$-square matrices with  entries in a commutative unital  ring $R$ is written $M_n(R)$. The identity matrix is $I$.  The group of invertible matrices is written $GL_n(R)$.
  The group of  matrices with determinant 1 is written $SL_n(R)$.

  A \emph{vector} is a column vector; $R^n$ is the set of column vectors. A  \emph{covector} is a row vector; $(R^n)^*$ is the set of row vectors.  For square matrices over fields,  \emph{eigenvector} means (right) eigenvector in the usual sense; an \emph{invariant subspace} of  a matrix $A$ is a vector space $V$ of column vectors such that  $AV\subseteq V$. An \emph{eigen-covector} (or \emph{left eigenvector}) of $A$ is a row vector $u\ne 0$ such that $uA=\la u$ for an eigenvalue $\la$; an \emph{invariant cosubspace} is  a vector space $U$ of row vectors such that  $UA\subseteq U$.  

  The letter $K$ always stands for a  field. 
  We write $K[A]$ for the unital subalgebra of $M_n(K)$ generated by $A$, so that $K[A]v$ is the $A$-invariant subspace generated by $v\in K^n$.

We write $V(x_1, \dots, x_n)$ for the Vandermonde matrix with $(i,j)$-entry $x_j^{i-1}$, and we write $$\delta(x_1, \dots, x_n)= \det V(x_1, \dots, x_n)= \prod_{s<t}(x_t-x_s)$$ for the Vandermonde determinant.

  A (multivariate) polynomial with integer coefficients is \emph{primitive} if the greatest common divisor of its coefficients is 1.  This is equivalent to saying that it is not the zero polynomial over any field.


  \section{The first triangulant}
  \subsection*{Definition and main theorem}
\begin{Def}Let $R$ be a commutative ring with 1. For two matrices \hbox{$A, B\in M_n(R)$,} we consider the matrix $M(A,B)\in M_{n^2}(R)$  consisting of $n\times n$ blocks, the $(i,j)$-th block being $A^{j-1}B^{i-1}$  $(i,j=1, \dots, n)$.  The \emph{(first) triangulant} of $A$ and $B$ is $T(A,B)=\det M(A, B)$.
\end{Def}

Observe that the $(i,j)$-th block of $M\left(A^\top, B^\top\right)$ is $$\left(A^\top\right)^{j-1}\left(B^\top\right)^{i-1}=(B^{i-1}A^{j-1})^\top,$$ whence $M\left(A^\top, B^\top\right)=M(B,A)^\top$ and thus $T\left(A^\top, B^\top\right)=T(B,A).$

Observe also that if the entries of $A$ and $B$ are viewed as (commuting) indeterminates, then $T(A, B)$ is a bihomogeneous polynomial with integer coefficients in those $n^2+n^2$ indeterminates. Each monomial appearing in $T(A,B)$ has degree $n\binom n 2$ in each of both sets of $n^2$ indeterminates.

Note that $T\left(P^{-1}AP,P^{-1}BP\right)=T(A,B) $  for all $P\in GL_n(R)$.

Let $K$ be a field and $A, B\in M_n(K)$.  Consider the following five properties.
\begin{enumerate}
\item[$(\bot)$]\label{co} $A$ has an eigen-covector $u$ and $B$ has an eigenvector $v$ such that $uv=0$.
\item[$(\subset)$]\label{inv} $B$ has an eigenvector contained in a one-codimensional invariant subspace of $A$.

\item[$(\triangle)$]\label{mtx}  $n\ge 2$, and there exists a matrix $P\in GL_n(K)$ such that $$P^{-1}AP=\begin{pmatrix} A' & *\\ 0 &a\end{pmatrix} \qquad  {\textrm {and}} \qquad  P^{-1}BP=\begin{pmatrix} b & *\\ 0 &B'\end{pmatrix},$$ where $a,b\in K$ and $A', B'\in M_{n-1}(K)$.

\item[$(<)$]\label{kicsi} $B$ has an eigenvector $v$ such that $\dim K[A]v<n$.

\item[$(\top)$]\label{nulla} $T(A, B)=0$.
\end{enumerate}

The name `triangulant' is justified by the relation between properties $(\triangle)$ and $(\top)$. The main result of this section is

\begin{Th}\label{1}

\begin{enumerate}
\item[(a)] $(\bot)\Longleftrightarrow(\subset)\Longleftrightarrow(\triangle)\implies(<)\implies(\top)$

\item[(b)] If $K$ contains all eigenvalues of $A$, then the first four properties are equivalent.

\item[(c)]  If $K$ contains all eigenvalues of $A$ and $B$, then all five  properties are equivalent.
\end{enumerate}
\end{Th}

\begin{proof}
(a) $(\bot)\Longleftrightarrow(\subset)$: The one-codimensional invariant subspaces of $A$ are precisely the subspaces of the form $u^\bot=\{v\in K^n:uv=0\}$, where $u$ is an eigen-covector of $A$. 

The implications $(\subset)\Longleftrightarrow(\triangle)$ and  $(\subset)\implies(<)$ are easy exercises left to the reader.

$(<)\implies(\top)$: The $n$ vectors  $v$, $Av$, $A^2v$, \dots, $A^{n-1}v$ are contained in $K[A]v$ and are therefore linearly dependent.
Thus, $$\sum_{j=1}^n a_jA^{j-1}v=0$$ holds with some  coefficients $a_j\in K$, not all zero. 


Furthermore,  $v$ is an eigenvector of $B$, i.e.,  $v\ne 0$ and $Bv=cv$ for some $c\in K$.
Therefore, we have  $$\sum_{j=1}^{n}A^{j-1}B^{i-1}a_jv=\sum_{j=1}^{n}A^{j-1}c^{i-1}a_jv=c^{i-1}\sum_{j=1}^{n}a_jA^{j-1}v=0$$ for all $i=1, \dots, n$.
This means that the $n^2$-dimensional vector formed by concatenating the $n$ vectors $a_1v$, \dots, $a_nv$ is a nonzero vector in the kernel of $M(A, B)$, whence  $T(A, B)=\det M(A,B)=0$.

(b) Assuming $(<)$, the $A$-invariant proper subspace $K[A]v$  is contained in a one-codimensional $A$-invariant subspace, proving $(\subset)$.

(c) Assuming $(\top)$, there exist vectors $v_1, \dots, v_n\in K^n$, not all zero, such that $$\sum_{j=1}^nA^{j-1}B^{i-1}v_j=0$$  for all $i=1, \dots, n$. This means that $$\sum_{j=1}^nA^{j-1}f(B)v_j=0$$ for all polynomials $f\in K[x]$. Let $f$ be a divisor of the minimal polynomial $m_B$ such that the vectors $f(B)v_j\in K^n$ are not all zero and $f$ has maximal degree under this property. Let $\lambda \in K$ be  a root of the quotient polynomial $m_B/f$.  Then all nonzero vectors $f(B)v_j$ are eigenvectors of $B$ with eigenvalue $\lambda $. If they are all scalar multiples of an eigenvector $v$, then \hbox{$\dim K[A]v<n$,} establishing $(<)$. Otherwise, the eigenspace of $B$ corresponding to the eigenvalue $\lambda $ has dimension $>1$ and therefore nontrivially intersects any one-codimensional invariant subspace of $A$, establishing $(\subset)$.
\end{proof}
The case when $B$ has an eigenvalue with geometric multiplicity $>1$, appearing in the last step of the proof of Theorem~\ref{1}, deserves more detailed analysis. This will be useful in Section~\ref{higher}, when we study triangulants of higher order.

\begin{Lemma}\label{2dim}
If $K$ is a field and $A, B\in M_n(K)$,  then $$\dim\ker M(A, B)\ge n\sum (m-1),$$ where $m$ runs over the  geometric multiplicities of distinct eigenvalues of $B$.
\end{Lemma}

\begin{proof}Let $V\le K^n$ be an eigenspace of $B$. 
 Consider the \hbox{$K^{n^2}\to K^n$} linear map whose matrix has block form $\begin{pmatrix}I & A & A^2 &\cdots & A^{n-1}\end{pmatrix}$, and restrict it to the subspace $V^{\oplus n}\le (K^n)^{\oplus n}=K^{n^2}$.  The  kernel of this restriction  has dimension $$\ge \dim V^{\oplus n}-\dim K^n=n(\dim V-1).$$ The  direct sum of all such kernels, where $V$ runs over all eigenspaces of $B$, is contained in the kernel of $M(A,B)$.
\end{proof}

\subsection*{Diagonalizable matrices} We wish to factorize the triangulant in the important special case when $B$ is a diagonal matrix.  For an $n$-square matrix  $A$, let $M_t(A)$  be the matrix whose columns are the $t$-th columns of $I$, $A$, \dots, $A^{n-1}$. Let $$\Delta_t(A)=\det M_t(A); \qquad \Delta(A)=\prod_{t=1}^n\Delta_t(A).$$

\begin{Pro}\label{diag}If $B=\diag(b_1, \dots, b_n)$, then $$T(A,B)=(-1)^{\lfloor n/2\rfloor}\Delta(A)\delta(b_1, \dots, b_n)^n.$$
\end{Pro}

\begin{proof}Permute the columns of the $n^2\times n^2$ matrix $M(A,B)$ by the involution  $qn+r+1\leftrightarrow rn+q+1$, where \hbox{$q,r\in\{0,1,\dots, n-1\}$.} This involution is  a product of $\binom n2$ (disjoint) transpositions, and  $\binom n2\equiv \lfloor n/2\rfloor \mod 2$, producing the sign that appears in the Proposition. The resulting matrix is a  product $(V(b_1, \dots, b_n)\otimes I)\diag(M_1(A), \dots, M_n(A))$, where the first factor is  a Kronecker product and the second factor is a block diagonal matrix. We then have $$T(A,B)=\det M(A, B)=(-1)^{\lfloor n/2\rfloor}\det V(b_1, \dots, b_n)^n\prod_{t=1}^n\det M_t(A)$$ and the Proposition follows.
\end{proof}

The polynomial $\Delta(A)$ was defined as a product of $n$ factors. When $A$ is diagonalizable, those factors can be themselves  factorized:

\begin{Lemma}\label{Delta} For the matrix ${A=P^{-1}\Lambda  P,}$ where $\Lambda =\diag(a_1, \dots, a_n)$ and   \hbox{$P=(p_{st})\in GL_n(K)$,}  we have
\begin{equation}\label{Deltat}\Delta_t(A
)=\delta(a_1, \dots, a_n)\frac1{\det P}\prod_{s=1}^np_{st}\qquad (t=1, \dots, n);\end{equation}
\begin{equation}\label{Deltaprod}\Delta(A
)=\delta(a_1, \dots, a_n)^n\frac1{\det P^n}\prod_{s=1}^n\prod_{t=1}^np_{st}.\end{equation}
\end{Lemma}

\begin{proof}
 The $t$-th column of $\Lambda^jP$ is $$\left(a_1^jp_{1t}, \dots, a_n^jp_{nt}\right)^\top=:v_j.$$ 
 Therefore, the $t$-th column of the matrix $A^j=P^{-1}\Lambda^j  P$  is the vector $P^{-1}v_j$. This is the $(j+1)$-th column of $M_t(A)$, whence we have $$M_t(A)=P^{-1}\cdot\begin{pmatrix} v_0& \cdots& v_{n-1}\end{pmatrix}.$$  Thus, $$\Delta_t(A)=\det M_t(A)=\det P^{-1}\cdot\det \begin{pmatrix} v_0& \cdots& v_{n-1}\end{pmatrix}$$ and the identity \eqref{Deltat} follows.

The identity \eqref{Deltaprod}  is immediate from \eqref{Deltat}.
\end{proof}
This  yields a product formula for the triangulant of two (non-simul\-ta\-ne\-ous\-ly) diagonalizable matrices:

\begin{Cor}\label{diagdiag} For $A$ as in Lemma~\ref{Delta} 
and $B=\diag(b_1, \dots, b_n)$, we have $$T(A, B)=(-1)^{\lfloor n/2\rfloor}\delta(a_1, \dots, a_n)^n\delta(b_1, \dots, b_n)^n\frac1{\det P^n}\prod_{s=1}^n\prod_{t=1}^np_{st}.$$  
\end{Cor}

\begin{proof}Immediate from   Proposition~\ref{diag} and Lemma~\ref{Delta}.
\end{proof}

\subsection*{Simultaneous triangularizability in dimension 2}
Recall that  the \emph{commutator} of matrices $A$ and $B$ is $[A,B]=AB-BA$.
\begin{Pro}\label{2} When $n=2$, we have $$T(A,B)=\det[A,B]=\tr A[A, B]B.$$
\end{Pro}
\begin{proof}All three  sides of this identity are polynomials, and are invariant under simultaneous conjugation of $A$ and $B$, so when verifying it, one may assume that $B=\diag(b_1, b_2)$, and use Proposition~\ref{diag} to obtain that $$T(A, B)=a_{12}a_{21}(b_1-b_2)^2.$$ On the other hand, $$[A,B]=\begin{pmatrix} 0 &a_{12}(b_2-b_1)\\ a_{21}(b_1-b_2) & 0\end{pmatrix}$$ and the Proposition follows.

Alternatively, the first equality follows from the identity $$M(A,B)=\begin{pmatrix}I &A\\B&AB\end{pmatrix}=\begin{pmatrix}I &0\\B&[A,B]\end{pmatrix}\begin{pmatrix}I &A\\0&I\end{pmatrix},$$ and the second equality follows from the identities $\tr[A,B]=0$, \begin{equation}\label{matyi}\tr[A,B]^2=\tr [A,B]AB-\tr[A,B]BA=-2\tr A[A,B]B,\end{equation} and $\tr^2X=\tr X^2+2\det X$ applied to $X=[A,B]$.
\end{proof}
 The identity~\eqref{matyi} was pointed out to us by M\'aty\'as Domokos.
\begin{Cor} For $2\times 2$ matrices $A$ and $B$ over an algebraically closed field, the following are equivalent:
\begin{enumerate}
    \item
$A$ and $B$ have  a common eigenvector;
\item $\det[A,B]=0$;
\item $\tr A[A, B]B=0$;
\item $[A,B]^2=0$.
\end{enumerate}
\end{Cor}

\begin{proof} The equivalence of the first three properties is immediate from Theorem~\ref{1}(c) and Proposition~\ref{2}.

(2)$\Longleftrightarrow$(4)  follows from the Cayley--Hamilton identity $$X^2-(\tr X)X+(\det X)I=0$$ applied to $X=[A,B]$, together with the identity $\tr[A,B]=0$.
\end{proof}

Note that the implication (2)$\implies$(1) also follows from Laffey's Theorem~\cite{L, RR}: if, for  the $n$-square matrices $A$ and $B$ over an algebraically closed field, the commutator $[A,B]$ has rank (at most) 1, then $A$ and $B$ are simultaneously triangularizable.

\subsection*{Unitary matrices}
Let us now return to the case of a general dimension $n$, but specialize to the ground field $K=\C$. 

Recall that two orthonormal bases $u_1$, \dots, $u_n$  and $v_1$, \dots, $v_n$ of $\C^n$  are \emph{mutually unbiased} if $|u_s^*v_t|^2=1/n$ for all $s$ and $t$. Mutually unbiased bases were introduced by Schwinger~\cite{Sch} in 1960, and play an important role in quantum science.

\begin{Pro}For any two $n\times n$ unitary matrices $A$ and $B$, we have $$|T(A,B)|\le n^{n^2/2},$$ with equality if and only if $A$ and $B$ have no multiple eigenvalues, $A^n$ and $B^n$ are scalar matrices, and the orthonormal eigenbases  of $A$ and $B$ are mutually unbiased.     
\end{Pro}

\begin{proof}Since unitary matrices are unitarily diagonalizable, we may assume that $A$ and $B$ are as in Corollary~\ref{diagdiag}, with $|a_s|=|b_s|=1$  for all $s$ and with $P$ unitary.
  Let $\epsilon
  $ be  a  primitive $n$-th root of unity.  Then $$|\delta(a_1, \dots, a_n)|\le\left|\delta\left(1, \epsilon, \dots, \epsilon^{n-1}\right)\right|=n^{n/2},$$ with equality if and only if $\{a_1, \dots, a_n\}=\left\{\rho, \rho\epsilon, \dots, \rho\epsilon^{n-1}\right\}$ for some $|\rho|=1$.  The analogous statement holds  for $\delta(b_1, \dots, b_n)$ as well.  Furthermore, $$\prod_{s=1}^n|p_{st}|\le\left(\frac1n\sum_{s=1}^n|p_{st}|^2\right)^{n/2}=n^{-n/2}$$ for each $t$, with equality if and only if $|p_{st}|^2=1/n$  for all $s$.  The Proposition now follows from Corollary~\ref{diagdiag}.
\end{proof}
Clearly, any orthonormal basis in $\C^n$ is the eigenbasis of some unitary matrix whose eigenvalues are the $n$-th roots of unity, each with multiplicity 1.
Thus, the notorious unsolved problem of determining the maximal number of mutually unbiased bases in a given dimension $n$ is equivalent to finding the maximal number of $n\times n$ unitary matrices such that their pairwise triangulants have absolute value $n^{n^2/2}$.
\section{Triangulants of higher order}\label{higher}
Throughout this section, $0\le k\le n$ are integers. The set $\{1,\dots, n\}$ is denoted by  $[n]$.  The set of $k$-element subsets of $[n]$ is denoted by $\binom{[n]} k$. It is considered as an ordered set with the lexicographic ordering.
\subsection*{Discriminants of higher order}
If $\la_1$, \dots, $\la_n$ are elements of a commutative ring, and $S\subseteq[n]$, then we write $$\la_S=\sum_{s\in S}\la_s.$$ For $r\ge 1$, we define the polynomial $\delta_r(\la_1, \dots, \la_n)$ to be the product
\begin{align*}
\prod\left(\la_T-\la_S: S, T\in{[n]\choose r}, S\cap T=\emptyset, \min S<\min T\right).\end{align*}
Then the Vandermonde determinant $\delta_1=\delta$ is an alternating polynomial, but $\delta_r$ for $r\ge 2$ is a  symmetric polynomial.

Using the elementary identity $$\la_T-\la_S=\la_{T\setminus S}-\la_{S\setminus T},$$ we obtain
$$\delta\left(\la_S: S\in{[n]\choose k}\right)=\prod_{r=1}^k \delta_r(\la_1, \dots, \la_n)^{{n-2r \choose k-r}}$$ and therefore \begin{equation}\label{kdelta}\delta\left(\la_S: S\in{[n]\choose k}\right)^{\binom nk}=\delta(\la_1, \dots, \la_n)^{\binom nk{n-2 \choose k-1}}\gamma_k(\la_1, \dots, \la_n),\end{equation} where $$\gamma_k(\la_1, \dots, \la_n)=\prod_{2\le r\le k} \delta_r(\la_1, \dots, \la_n)^{\binom nk{n-2r \choose k-r}}$$ is a  symmetric polynomial.  Observe that $\gamma_k=\gamma_{n-k}$.

If the matrix $A\in M_n(K)$ has  characteristic polynomial \begin{equation}\label{charpol}\det (xI-A)=\prod_{s=1}^n(x-\la_s),\end{equation} where the $\la_s$ are in an extension of $K$,
 then we write  $$D(A)=\delta(\la_1, \dots, \la_n)^2=\prod_{s<t}(\la_t-\la_s)^2,$$ $$D_1(A)=\begin{cases} D(A) \qquad (\cha K\ne 2)\\\sqrt{D(A)}\quad (\cha K= 2)\end{cases} $$
and, for $r\ge 2$, we write $$D_r(A)=\delta_r(\la_1, \dots, \la_n).$$
These expressions are symmetric in $\la_1$, \dots, $\la_n$, so we can think of $D_r$  $(r\ge 1)$  
as a conjugation-invariant  polynomial, with integer coefficients,  in the $n^2$ entries of $A$, which are now viewed as indeterminates. 

The polynomial $D_r$  is homogeneous of degree $n(n-1)$ if $r=1$ and $\cha K\ne 2$, and of degree $
{n\choose 2r}{2r\choose r}/2$ otherwise.
In particular, we have $D_r=1$ for $r>\lfloor n/2\rfloor$. The polynomials $D_r$ for $r=1, \dots,\lfloor n/2\rfloor $ are non-constant, and no two of them coincide (not even up to a scalar multiplier). 
\begin{Lemma}\label{Dirred}For each $1\le r\le\lfloor n/2\rfloor$, the polynomial $D_r$ is irreducible in the polynomial ring  $K[a_{ij}:i,j=1, \dots, n]$ for any field $K$, and thus also in  $\Z[a_{ij}:i,j=1, \dots, n]$.
\end{Lemma}

\begin{proof}It suffices to prove irreducibility over $K$. We may assume that $K$ is algebraically closed. 

Assume that a  polynomial $f\in K[a_{ij}:i,j=1, \dots, n]$  divides $D_r$.  Then the transformed polynomial $(Pf)(A):=f\left(P^{-1}AP\right)$ divides \hbox{$D_r\left(P^{-1}AP\right)=D_r$} for all  $P\in GL_n(K)$. But $D_r$ has only finitely many divisors, up to scalar multipliers. Since $GL_n(K)$ is a connected algebraic group, we must have $Pf=\chi(P)f$ for  some multiplicative character $\chi:GL_n(K)\to K^\times$. Since $SL_n(K)$ is a perfect group, we have $\chi(P)=1$ and thus $Pf=f$ for all $P\in SL_n(K)$. Conjugation by scalar matrices is the identity, so in fact $Pf=f$ for all $P\in GL_n(K)$, i.e., $f$ is a symmetric polynomial in the eigenvalues, and so is $D_r/f$. From the definition of $D_r$, we see that $f$ must be 1 or $D_r$, up to a scalar multiplier.
\end{proof}
We define $$G_k(A)=\prod_{2\le r\le k} D_r(A)^{\binom nk{n-2r \choose k-r}}=\gamma_k(\la_1, \dots, \la _n).$$
\subsection*{The $k$-th triangulant}
Let $A\in M_n(K)$. We let $A$ act on the $n\choose k$-dimen\-sion\-al vector space $\bigwedge^kK^n$ of exterior tensors via the Leibniz rule \begin{equation}\label{A_k}A_k(v_1\wedge\dots\wedge  v_k)
=\sum_{i=1}^k v_1\wedge\dots\wedge Av_i\wedge\dots\wedge  v_k.\end{equation}
The linear transformation $A_k$ is given by an $\binom nk$-square matrix whose entries are linear forms, with integer coefficients, in the entries of $A$. 

Observe that $A_k^\top=\left(A^\top\right)_k$.

 If $B=\diag(b_1, \dots, b_n)$, then $$B_k=\diag\left(b_S:S\in{[n]\choose k}\right),$$ and, from Proposition~\ref{diag}, we get
 \begin{align*}T(A_k,B_k)=(-1)^{\left\lfloor\binom nk/2\right\rfloor}\Delta(A_k)\delta\left(b_S: S\in{[n]\choose k}\right)^{n\choose k}.\end{align*}
 Here $\Delta(A_k)$ is a homogeneous polynomial of degree ${n\choose k}\binom{\binom nk}2$, with integer coefficients,  in the entries of $A_k$, and therefore also in the entries of  $A$. 

Using formula~\eqref{kdelta}, we obtain \begin{equation}\label{kdiag}
T(A_k,B_k)=(-1)^{\left\lfloor\binom nk/2\right\rfloor}\Delta(A_k)\delta(b_1, \dots, b_n)^{{n\choose k}{n-2 \choose k-1}}G_k(B)
\end{equation}  for $B=\diag(b_1, \dots, b_n)$. 

We aim to show that the polynomial $T(A_k,B_k)$ is divisible by the polynomial $G_k(B)$ for general (non-diagonal) $B$. This will follow from Lemma~\ref{2dim} and the following general statement.
\begin{Lemma}\label{general}
Let $R$ be a unique factorization domain, $p\in R$ an irreducible element, and $M$ a square matrix with entries in $R$.
Let $F_p$ be  the field of fractions of the integral domain $R/(p)$, and let $\bar M$ be the matrix $M$ when viewed as a matrix over $F_p$. Denote $d={\dim\ker\bar M}$. Then $p^d$ divides $\det M$.

\end{Lemma}
\begin{proof}
We can choose $d$ columns of $\bar M$ such that each chosen column is a  linear combination, over $F_p$,  of the non-chosen columns. Thus, we can choose $d$ columns of $M$ such that each chosen column, if multiplied by a suitable element of $R$ that is not divisible by $p$, becomes congruent mod $p$ to a linear combination, over $R$,  of the non-chosen columns. Thus, $\det M$ can be multiplied by an element of $R$, not divisible by $p$, to get the determinant of a matrix containing $d$ columns all of whose entries are divisible by $p$.
\end{proof}
\begin{Pro} The $2n^2$-variate polynomial $T(A_k,B_k)$ is divisible by the product $G_k(A)G_k(B)$ in the ring $R= \Z[a_{ij}, b_{ij}: i,j=1, \dots, n]$, where \hbox{$A=(a_{ij})_{11}^{nn}$} and $B=(b_{ij})_{11}^{nn}$.
\end{Pro}
\begin{proof} 
The eigenvalues of $A$ and $A^\top$ coincide. Thus, for all $r\ge 1$, we have $D_r(A)=D_r\left(A^\top\right)$. Therefore, \hbox{$G_k(A)=G_k\left(A^\top\right)$}. Furthermore,  $$T(A_k, B_k)=T\left(B_k^\top, A_k^\top\right)=T\left(\left(B^\top\right)_k, \left(A^\top\right)_k\right).$$ Thus, it suffices to prove the divisibility by $G_k(B)$ and then it will follow for $G_k(A)$.

 Fix $r\ge 2$, and let $B\in M_n\left(\C\right)$ have  $D(B)\ne 0=D_r(B)$.   Then $B$ has $n$ distinct eigenvalues and therefore is diagonalizable. Thus, $B_k$ is also diagonalizable. Since $D_r(B)=0$, it follows that $B_k$ has ${n-2r \choose k-r}$ disjoint pairs of equal eigenvalues, whence, by Lemma~\ref{2dim}, we  have $$\dim\ker M(A_k, B_k)\ge\binom nk \binom{n-2r}{k-r}=:d,$$ i.e., all minors of $M(A_k, B_k)$ of dimension $>\binom nk^2 -d$ are zero.

 Thus, any such minor, when viewed as a polynomial in the entries of $A$ and $B$, is divisible by $D_r(B)$ over $ \C$, but then also over $\Z$ because all these polynomials have integer coefficients and $D_r(B)$ is primitive.  We may now apply Lemma~\ref{general} with 
  $p=D_r(B)$ and $M=M(A_k, B_k)$ to deduce that $D_r(B)^d$ divides $\det M=T(A_k, B_k)$.

  This is true for all $r\ge 2$. By Lemma~\ref{Dirred}, the polynomials $D_r(B)$, where $r\le\lfloor n/2\rfloor$, are (distinct) irreducibles. The Proposition follows.
\end{proof}

\begin{Def}Let $0\le k\le n$.  The \emph{$k$-th triangulant} of two $n$-square matrices $A$ and $B$ is $$T_k(A,B)=\frac{T(A_k, B_k)}{G_k(A)G_k(B)}.$$
\end{Def}
This is a bihomogeneous polynomial with integer coefficients. It has degree ${n\choose 2}{n\choose k}{n-2\choose k-1}$ in each of both sets of $n^2$ variables. 

Observe that \hbox{$T_1(A,B)=T(A,B)$.}

 It is not difficult to see that $T_{n-k}(A,B)= T_k(B,A)$. In particular, we have \hbox{$T_0=T_n=1$.}

As we have done for $k=1$, it is important to obtain product formulas  for $T_k(A, B)$ when $A$ and/or $B$ are diagonalizable.

When $B=\diag(b_1, \dots, b_n)$, we have, from formula~\eqref{kdiag},
\begin{equation}\label{Tkdiag}G_k(A)T_k(A,B)=(-1)^{\left\lfloor{\binom nk}/2\right\rfloor}\Delta(A_k)\delta(b_1, \dots, b_n)^{{n\choose k}{n-2 \choose k-1}}.\end{equation}

 If $A=P^{-1}\Lambda P$ with $\Lambda =\diag(a_1, \dots, a_n)$, then, from Lemma~\ref{Delta} and formula~\eqref{kdelta}, we obtain
\begin{equation}\label{kDelta}\Delta(A_k)=\delta(a_1, \dots, a_n)^{{n\choose k}{n-2 \choose k-1}}G_k(A)\frac{\prod_{S\in \binom{[n]}k}\prod_{T\in \binom{[n]}k}\det P[S|T]}{\det P^{\binom nk\binom{n-1}{k-1}}}.\end{equation}

From formulas~\eqref{Tkdiag} and \eqref{kDelta}, or, alternatively, from Corollary~\ref{diagdiag} together with formula~\eqref{kdelta}, 
 we obtain the following formula for the $k$-th triangulant of two (non-simul\-ta\-ne\-ously) diagonalizable matrices:

 \begin{Pro}\label{kdiagdiag}If $\Lambda =\diag(a_1, \dots, a_n)$, $P\in GL_n(K)$,  $A=P^{-1}\Lambda P$, and \hbox{$B=\diag(b_1, \dots, b_n)$}, then the $k$-th triangulant $ T_k(A,B)$ is given by the product $$
(-1)^{\left\lfloor{\binom nk}/2\right\rfloor}(\delta(a_1, \dots, a_n)\delta(b_1, \dots, b_n))^{{n\choose k}{n-2 \choose k-1}}\frac{\prod_{S\in \binom{[n]}k}\prod_{T\in \binom{[n]}k}\det P[S|T]}{\det P^{\binom nk\binom{n-1}{k-1}}}.
 $$
\end{Pro}
This might suggest the false impression that $T_k(A,B)$, as  a polynomial, is divisible by $D_1(A)D_1(B)$. We shall now refute this. Note that if  $B$ has characteristic polynomial $$\det (xI-B)=\prod_{s=1}^n(x-b_s),$$
 then the characteristic polynomial of $B_k$ is  $$\det (xI-B_k)=\prod_{S\in{[n]\choose k}}(x-b_S).$$

\begin{Pro}Consider the $2n^2$-variate polynomial $$T_k(A,B)\in K[a_{ij}, b_{ij}: i,j=1, \dots, n],$$ where $A=(a_{ij})_{11}^{nn}$ and $B=(b_{ij})_{11}^{nn}$.
 For $1\le r\le \lfloor n/2\rfloor$, the discriminants $D_r(A)$ and $D_r(B)$ do not divide $T_k(A,B)$.

\end{Pro}
\begin{proof}
For all $r\ge 1$, we have $$D_r(A)=D_r\left(A^\top\right)\quad \textrm{ and }\quad T_k(A, B)=T_k\left(B^\top, A^\top\right).$$ Thus, it suffices to prove the Proposition  for $D_r(B)$ and then it will follow for $D_r(A)$.

 We may assume that $K$ is algebraically closed.

Case $r=1$:  Let $B$ have $n-1$ Jordan blocks of dimensions 2, 1, \dots, 1,  with \hbox{$n-1$}  distinct eigenvalues such that $G_k(B)\ne 0$. Then $D_1(B)=0$ because $B$ has a double eigenvalue. On the other hand, all eigenvalues of $B_k$  have geometric multiplicity 1, and all eigenvectors of $B_k$ belong to the standard basis of $\bigwedge^kK^n$.

Choose a  diagonal matrix $B'$  such that $D(B')G_k(B')\ne 0$. From Proposition~\ref{kdiagdiag}, we see that $T_k(-, B')$ is not the zero polynomial. Thus, $$T(-_k, B'_k)=G_k(-)G_k(B')T_k(-, B')$$ is also not the zero polynomial. But every eigenvector of $B_k$ is an eigenvector of $B'_k$, so, using Theorem~\ref{1}(c), equivalence  $(\top)\Leftrightarrow(\bot)$, we infer that $T(-_k, B_k)$ is not the zero polynomial. Thus, $T_k(-, B)$ is not the zero polynomial.


Case $r\ge 2$:     Immediate from Proposition~\ref{kdiagdiag}. 
\end{proof}
\subsection*{Invariant subspaces}
So far, we have studied the action $A_k$, defined by formula~\eqref{A_k}, of  \hbox{$A\in M_n(K)$} on the $k$-th exterior power $\bigwedge^k K^n$. In fact, $A_k$ is only one of $k$ natural actions. 
For $A\in M_n(K)$, write $$(I+xA)^{\otimes k}=I+\sum_{i=1}^k x^iA_k^{(i)}.$$ The linear transformations $A_k^{(i)}$ act on $\bigotimes^k K^n$. They leave the kernel of the natural quotient map $\bigotimes^kK^n\to\bigwedge^kK^n$ invariant, so they act on $\bigwedge^kK^n$. 
For example, the action of $A_k^{(1)}$ on $\bigwedge^k K^n$ is the linear transformation $A_k$.

The linear transformations $A_k^{(i)}$ can be used  for a characterization of  $k$-dimensional invariant subspaces of $A$, to which we now turn. 

 Recall the following bit of notation: for  a field extension $L|K$ and a $K$-vector space $V$, we write \hbox{$V_L=L\otimes_K V$} for the  $L$-vector space obtained by extension of scalars.
 
\begin{Lemma}\label{genfct}  Let $A\in M_n(K)$, and let $V\le K^n$ be a $k$-dimensional subspace.  The following are equivalent.
\begin{enumerate}
\item $AV\subseteq V$
\item The one-dimensional subspace $\bigwedge ^k V\le \bigwedge^k K^n$ is invariant under the action of  each $A_k^{(i)}$, $i=1, \dots, k$.
\item\label{karpolinv}  The one-dimensional subspace $\bigwedge^kV_{K(x)}\le \bigwedge ^k K(x)^n$ is  invariant under the action of $(I+xA)^{\otimes k}$ .
\end{enumerate}
 If these properties hold, then   the eigenvalue corresponding to~\eqref{kinv} is the sum $e_{i,V}(A)$ of all $\binom ki$ diagonal $i\times i$ minors of $A\left|_V\right.$, and the eigenvalue corresponding to~\eqref{karpolinv} is \hbox{$\det(I+xA)\left|_{V_{K(x)}}\right.$}.
\end{Lemma}

\begin{proof}The implications (1)$\implies$(2)$\implies$(3) are easy to see. 

For (3)$\implies$(1),
observe that the action of $(I+xA)$ 
 on $
 K(x)^n$ is invertible because its determinant is 
  a nonzero polynomial. Let  $$W=(I+xA)V_{K(x)}\le K(x)^n,$$ a $k$-dimensional subspace. 
Then
we have $$(I+xA)^{\otimes k}\bigwedge\nolimits ^k V_{K(x)}=
\bigwedge \nolimits^k W.$$ This one-dimensional subspace is the same as  $\bigwedge\nolimits ^k V_{K(x)}$ if and only if $W$ is the same as $V_{K(x)}$; equivalently, $V_{K(x)}$ is invariant under $I+xA$; equivalently, under $A$. The implication (3)$\implies$(1) follows. 

The claim about the eigenvalue corresponding to~\eqref{karpolinv} is easy to verify. The similar claim for~\eqref{kinv} then follows by the identity $$\det(I+xA)\left|_{V_{K(x)}}\right.=1+\sum_{i=1}^ke_{i,V}(A)x^i.$$
\end{proof}

\subsection*{Projective varieties} \it From this point on, we  assume throughout the paper  that $K$ is algebraically closed. \rm

 We write $\PP V$ for the projectivization of  a vector space $V$, and  we write $\Gr_kV$ for the Grassmannian variety whose points parametrize $k$-dimensional linear subspaces of $V$, so that $\Gr_1V=\PP V$. The notation $\Gr_kK^n$ is abbreviated to $\Gr_k(n)$ when $K$ is understood. We write $\Gr_k^*(n)=\Gr_k(K^n)^*$. 
 
Consider the \emph{Pl\"ucker embedding} $$\iota: \Gr_k(n)\to \PP^{\binom nk -1}$$ defined by $\iota (V)=\bigwedge ^k V$. The homogeneous coordinates of $\iota(V)$ are called the \emph{Pl\"ucker coordinates} of $V$. We identify $\Gr_k(n)$ with its image under $\iota$, which is a projective variety.

\begin{Cor}\label{invar} The set  $$\{([A], V)\in \PP M_n(K)\times \Gr_k(n)\; :\; AV\subseteq V\}$$ is a projective variety.
\end{Cor}

\begin{proof}First let $k=1$.   The $2\times 2$ minors of the $n\times 2$ matrix $\begin{pmatrix}v & Av\end{pmatrix}$ are homogeneous polynomials in the entries of the matrix $A$ and the vector $v$. Their simultaneous vanishing defines the point set in question, proving the assertion.

For general $k$,   Lemma~\ref{genfct} shows that $AV\subseteq V$ holds if and only if \hbox{$A_k^{(i)}\iota(V)\subseteq\iota(V)$} for each $i=1, \dots, k$. 
The entries of $A_k^{(i)}$ are  homogeneous polynomials in the entries of $A$.
 This reduces the statement to the case $k=1$ discussed before.
\end{proof}

\begin{Lemma}\label{degenerate}Let $0<k<n$.   Then the set $Z_k$ of pairs $$(U,V)\in\Gr_k^*(n)\times \Gr_k(n)$$ such that the row-column multiplication $U\times V\to K$ is degenerate is an irreducible projective variety.
\end{Lemma}
\begin{proof} Let the row vectors $u_i$ and the column vectors $v_j$ $(i,j=1, \dots, k)$ form bases of $U$ and $V$, respectively. Degeneracy occurs if and only if \hbox{$\det(u_iv_j)_{11}^{kk}=0$.} By the Cauchy--Binet formula, this the same as saying that $\iota (U)$ and $\iota(V)$ are orthogonal --- a  bilinear equation in the Pl\"ucker coordinates, whence we indeed have  a projective variety.

The natural action of the connected algebraic group $GL_n(K)$ on this variety has a Zariski dense orbit, proving irreducibility.
\end{proof}

\begin{Pro}\label{irredproj} Let $0<k<n$. Then the  set $X_k$ of pairs $$([A], [B])\in (\PP M_n(K))^2$$ such that there exists $(U,V)\in Z_k$ with $UA\subseteq U$ and $BV\subseteq V$ 
is an irreducible projective variety.
\end{Pro}

\begin{proof}  Let $Y_k$ be the set of quadruples $$([A], [B], U, V)\in(\PP M_n(K))^2\times \Gr_k^*(n)\times \Gr_k(n)$$  such that $U$ is an invariant cosubspace of $A$, $V$ is an invariant subspace of $B$, and the row-column multiplication $U\times V\to K$ is degenerate. By Corollary~\ref{invar} and Lemma~\ref{degenerate}, $Y_k$ is a projective variety. It is a $\PP ^d\times \PP ^d$-bundle over $Z_k$, where $d=n^2-kn+k^2-1$. Therefore, $Y_k$ is irreducible. Since $X_k$ is the image of $Y_k$ under the projection to $\PP M_n(K)^2$, the Proposition follows.
\end{proof}

\subsection*{Vanishing of higher-order triangulants} We are now ready for the main result of this paper: for two matrices, the existence of nontrivially intersecting invariant subspaces with given, complementary dimensions is characterized by the vanishing of a  higher-order triangulant. The name `$k$-th triangulant' is justified by the following generalization of Theorem~\ref{1}.   We continue to assume that  $K$ is algebraically closed.
\begin{Th}\label{k} For $A,B\in M_n(K)$, the following  are equivalent.
\begin{enumerate}
\item\label{kco} $A$ has an invariant cosubspace $U$ and $B$ has an invariant subspace  $V$ such that $\dim U=\dim V=k$ and the row-column multiplication $U\times V\to K$ is degenerate.
\item\label{kinv} $B$ has a $k$-dimensional invariant subspace that nontrivially intersects  a $k$-codimensional invariant subspace of $A$.


\item\label{kkicsi} There exists a vector $v\ne 0$ such that $\dim K[A]v\le n-k$ and $\dim K[B]v\le k$.

\item\label{knulla} $T_k(A, B)=0$.
\end{enumerate}
\end{Th}





\begin{proof} \eqref{kco}$\Longleftrightarrow$\eqref{kinv}:
 The $k$-codimensional invariant subspaces of $A$ are precisely the subspaces of the form $U^\bot=\{v\in K^n:Uv=0\}$, where $U$ is a $k$-dimensional invariant cosubspace of $A$. 

 \eqref{kinv}$\implies$\eqref{kkicsi}: Choose $v\ne 0$ in the intersection of the two invariant subspaces.


 \eqref{kkicsi}$\implies$\eqref{kinv}: The $A$-invariant  subspace $K[A]v$  is contained in a $k$-co\-dimensional $A$-invariant subspace, and  the $B$-invariant  subspace $K[B]v$  is contained in a $k$-dimensional $B$-invariant subspace.

\eqref{kco}$\implies$\eqref{knulla}: We must have  $0<k<n$.  On the irreducible projective variety $X_k$, the polynomials $G_k(A)$ and $G_k(B)$ do not vanish everywhere, as shown by choosing $A$ and $B$ to be generic diagonal matrices. Thus, we may assume that \hbox{$G_k(A)G_k(B)\ne 0$.}  From~\eqref{kco}, using Theorem~\ref{1}, we have  $$0=T(A_k, B_k)=G_k(A)G_k(B)T_k(A, B).$$ Therefore, $T_k(A,B)=0$ 
 as claimed.

 \eqref{knulla}$\implies$\eqref{kco}:  We wish to prove that $(A,B)\in X_k$.  Since $X_k$ is a projective variety, and the polynomial $T_k(A, B)$ is not divisible by $D_r(A)$ or $D_r(B)$ for $1\le r\le\lfloor n/2\rfloor$, we may assume the genericity condition \begin{equation*}
  D (A)D(B)G_k(A)G_k(B)\ne 0.\end{equation*} We have $T(A_k, B_k)=G_k(A)G_k(B)T_k(A, B)=0$  by
  ~\eqref{knulla}. Thus, $A_k$ has an eigen-covector $u$ and $B_k$ has an eigenvector $v$ such that \hbox{$uv=0$.}  Because of the genericity condition, we must have $u\in\bigwedge^k U$ and \hbox{$v\in\bigwedge^k V$} for some invariant (co)subspaces $U$ and $V$ as  in
  ~\eqref{kco}.
\end{proof}
In the last step of the proof of  of Theorem~\ref{1},  we have seen that if the matrix $B$ has an eigenvalue with geometric multiplicity $>1$, then \hbox{$T_1(A,B)=0$} for all $A$. A generalization of this was given  in Lemma~\ref{2dim}. Now we present a  different generalization together with its  converse.
\begin{Cor}\label{geom1} Let $0<k<n$. 

For  $ B\in M_n(K)$, the following are equivalent. \begin{enumerate}
\item  $B$ has an eigenvalue with geometric multiplicity $>1$;

\item $T_k(A, B)=0$ for all $A\in M_n(K)$.
\end{enumerate}
\end{Cor}

\begin{proof} As both properties are invariant under conjugation, we may assume that $B$ is in Jordan canonical form.

(1)$\implies$(2):  Any basis vectors that correspond to the first few rows of some Jordan blocks span  a $B$-invariant subspace of $K^n$. As $B$ has (at least) two Jordan blocks with the same eigenvalue, we can choose thus a $k$-dimensional invariant subspace containing at least one basis vector from one of these blocks (let $e$ be the last such basis vector), and not fully containing another one of these blocks (let $f$ be the first basis vector in this block not contained). If we replace the generator $e$ by any nonzero linear combination of $e$ and $f$, we still get a $k$-dimensional invariant subspace. Together, these subspaces form a $(k+1)$-dimensional subspace, so at least one of them intersects any $k$-codimensional invariant subspace of $A$ (for any $A$), establishing Property~\eqref{kinv} of Theorem~\ref{k}. By that theorem, we conclude that $T_k(A,B)=0$. 

(2)$\implies$(1): 
Choose $B'=\diag(b_1,\dots, b_n)$ with all $b_i$ distinct.

  Choose $A\in M_n(K)$ with $\Delta(A_k)\ne 0$. This is possible by formula~\eqref{kDelta}. Then, from formula~\eqref{Tkdiag}, we have $T_k(A, B')\ne 0$. From Theorem~\ref{k}, we see that Property~\eqref{kco} described in that theorem does not hold for $A$ and $B'$.

   Assume that  all eigenvalues of $B\in M_n(K)$ have geometric multiplicity $1$. Then all invariant subspaces of  $B$ are also invariant subspaces  of $B'$. Therefore,  Property~\eqref{kco} of Theorem~\ref{k} does not hold for $A$ and $B$. Applying Theorem~\ref{k} once again,  we have  $T_k(A, B)\ne 0$.
\end{proof}

\begin{Th}\label{irred}  For  $0<k<n$, the $2n^2$-variate bihomogeneous polynomial $T_k$ is irreducible over any field and thus also  over $\Z$.
\end{Th}
\begin{proof} From Theorem~\ref{k} and Hilbert's Nullstellensatz, the polynomial $T_k$ is, up to a scalar multiplier, a power of the irreducible polynomial defining the projective hypersurface $X_k$.  It must be the \it first \rm power, because the rational function appearing in Proposition~\ref{kdiagdiag} is not a perfect power for  $0<k<n$, as the polynomials $\det P[S|T]$ are distinct irreducibles.
\end{proof}

\section{An open problem}
 The polynomial $T_k(A,B)$ is invariant under simultaneous conjugation of $A$ and $B$ by any invertible matrix.  In characteristic zero, any such invariant polynomial can be rewritten as  a polynomial in  traces  of products involving $A$ and $B$ as factors~\cite{P,S}.  For prime characteristic, traces do not in general suffice, but  characteristic coefficients do~\cite{D}. Present $T_k$ in this way, generalizing Proposition~\ref{2}. This seems to be  challenging  even for $k=1$.



\begin{thebibliography}{99}
\bibitem{D}  S.\ Donkin, Invariants of several matrices, Invent.\ Math.\ 110 (1992), 389--401.
\bibitem{L} T.\ J.\ Laffey, Simultaneous triangularization of a pair of matrices-low
rank cases and the nonderogatory case, Linear and Multilinear Algebra
6 (1978), 269--306.
\bibitem{P} C.\ Procesi, The invariant theory of $n \times  n$ matrices, Adv.\ Math.\  19/3 (1976),  306--381.

\bibitem{RR} H.\ Radjavi, P.\ Rosenthal: Simultaneous triangularization, Springer, 2000.
\bibitem{Sch} J.\ Schwinger: Unitary Operator Bases, Proc.\ Nat.\ Acad.\ Sci.\ U.S.A. 46/4 (1960),  570--579.
\bibitem{S} K.\ S.\ Sibirski\v \i, Algebraic invariants of a system of matrices (Russian), Sibirsk.\ Mat.\ Z.\  9
(1968), 152--164.
 \end{thebibliography}
\end{document}